\documentclass[12pt]{amsart}
\usepackage{amscd,amsmath,amsthm,amssymb,enumerate}
\usepackage[left]{lineno}
\usepackage{pstcol,pst-plot,pst-3d}
\usepackage{pstricks}
\usepackage{stmaryrd}
\usepackage[utf8]{inputenc}
\usepackage{pstricks-add}
\usepackage{epstopdf}
\usepackage{graphicx}

\usepackage[all]{xy}

 \def\NZQ{\mathbb}               
 \def\NN{{\NZQ N}}

 \def\FF{{\NZQ F}}

 \def\frk{\mathfrak}               

 \def\mm{{\frk m}}

 \def\fkm{{\frk m}}

 \def\G{{\mathcal G}}

 \def\Soc{{\mathbf Soc}}

 \def\opn#1#2{\def#1{\operatorname{#2}}} 
 \opn\chara{char} \opn\length{\ell} \opn\pd{pd} \opn\rk{rk}
 \opn\projdim{proj\,dim} \opn\injdim{inj\,dim} \opn\rank{rank}
 \opn\depth{depth} \opn\grade{grade} \opn\height{height}
 \opn\bigheight{bigheight}
 \opn\embdim{emb\,dim} \opn\codim{codim}
 
 \opn\Tr{Tr} \opn\bigrank{big\,rank}
 \opn\superheight{superheight}\opn\lcm{lcm}
 \opn\trdeg{tr\,deg}
 \opn\reg{reg} \opn\lreg{lreg} \opn\ini{in} \opn\lpd{lpd}
 \opn\size{size} \opn\sdepth{sdepth}
 \opn\link{link}\opn\fdepth{fdepth}\opn\lex{lex}
 \opn\tr{tr}
 \opn\type{type}
 \opn\gap{gap}
 \opn\arithdeg{arith-deg}
 \opn\Deg{Deg}
 \opn\sat{sat}
 \opn\mat{mat}
 \opn\Mat{Mat}
 \opn\div{div} \opn\Div{Div} \opn\cl{cl} \opn\Cl{Cl}
 \opn\Spec{Spec} \opn\Supp{Supp} \opn\supp{supp} \opn\Sing{Sing}
 \opn\Ass{Ass} \opn\Min{Min}\opn\Mon{Mon}
 \opn\Ann{Ann} \opn\Rad{Rad} \opn\Soc{Soc}
 \opn\Im{Im} \opn\Ker{Ker} \opn\Coker{Coker} \opn\Am{Am}
 \opn\Hom{Hom} \opn\Tor{Tor} \opn\Ext{Ext} \opn\End{End}
 \opn\Aut{Aut} \opn\id{id}
 
 \opn\nat{nat}
 \opn\pff{pf}
 \opn\Pf{Pf} \opn\GL{GL} \opn\SL{SL} \opn\mod{mod} \opn\ord{ord}
 \opn\Gin{Gin} \opn\Hilb{Hilb}\opn\sort{sort}
 \opn\PF{PF}\opn\Ap{Ap}
 \opn\mult{mult}
 \opn\bight{bight}
 \opn\aff{aff}
 \opn\relint{relint} \opn\st{st}
 \opn\lk{lk} \opn\cn{cn} \opn\core{core} \opn\vol{vol}  \opn\inp{inp} \opn\nilpot{nilpot}
 \opn\link{link} \opn\star{star}\opn\lex{lex}\opn\set{set}
 \opn\width{wd}
 \opn\Fr{F}
 \opn\QF{QF}
 \opn\G{G}
 \opn\type{type}\opn\res{res}
 \opn\conv{conv}
 \opn\Shad{Shad}
 \opn\gr{gr}
 
 \def\pot#1#2{#1[\kern-0.28ex[#2]\kern-0.28ex]}

 \opn\dirlim{\underrightarrow{\lim}}
 \opn\inivlim{\underleftarrow{\lim}}
 \let\iso=\cong

 \let\to=\rightarrow
 
 \def\Implies{\ifmmode\Longrightarrow \else
         \unskip${}\Longrightarrow{}$\ignorespaces\fi}
 \def\implies{\ifmmode\Rightarrow \else
         \unskip${}\Rightarrow{}$\ignorespaces\fi}
 \def\iff{\ifmmode\Longleftrightarrow \else
         \unskip${}\Longleftrightarrow{}$\ignorespaces\fi}

 \let\:=\colon
\theoremstyle{plain}
 \newtheorem{Theorem}{Theorem}[section]
 \newtheorem{Lemma}[Theorem]{Lemma}
 \newtheorem{Corollary}[Theorem]{Corollary}
 \newtheorem{Proposition}[Theorem]{Proposition}
 \newtheorem{Claim}{Claim}

\theoremstyle{definition}
 \newtheorem{Remark}[Theorem]{Remark}
 \newtheorem{Example}[Theorem]{Example}

 \newtheorem{Question}[Theorem]{Question}
 \newtheorem{Notation}[Theorem]{Notation}
 \newtheorem*{Acknowledgment}{Acknowledgment}
 \let\epsilon\varepsilon
 \let\kappa=\varkappa
\opn\sign{Sign}

\title[Certain monomial ideals]{Certain monomial ideals whose numbers of generators of powers descend}
\author{Reza Abdolmaleki}
\address{Department of Mathematics, Institute for Advanced studies in basic science (IASBS), 45195-1159 Zanjan, Iran}
\email{abdolmaleki@iasbs.ac.ir}
\author{Shinya Kumashiro}
\address{Department of Mathematics and Informatics, Graduate School of Science and Engineering, Chiba University, Yayoi-cho 1-33, Inage-ku, Chiba, 263-8522, Japan}
\email{caxa2602@chiba-u.jp}

\thanks{2020 {\em Mathematics Subject Classification.} 13D40, 13F20}
\thanks{{\em Key words and phrases.} number of generators, polynomial ring, monomial ideal}
\thanks{The second author was supported by JSPS KAKENHI Grant Number JP19J10579 and JSPS Overseas Challenge Program for Young Researchers.}

\begin{document}

\begin{abstract}
This paper studies the numbers of minimal generators of powers of monomial ideals in polynomial rings. 
For a monomial ideal $I$ in two variables, Eliahou, Herzog, and Saem gave a sharp lower bound $\mu (I^2)\ge 9$ for the number of minimal generators of $I^2$ with $\mu(I)\geq 6$. Recently, Gasanova constructed monomial ideals such that $\mu(I)>\mu(I^n)$ for any positive integer $n$. In reference to them, we construct a certain class of monomial ideals such that $\mu(I)>\mu(I^2)>\cdots >\mu(I^n)=(n+1)^2$ for any positive integer $n$, which provides one of the most unexpected behaviors of the function $\mu(I^k)$.  The monomial ideals also give a peculiar example such that the Cohen-Macaulay type (or the index of irreducibility) of $R/I^n$ descends. 
\end{abstract}

\maketitle



\section{Introduction}\label{section1}

The number of minimal generators of ideals plays an important role to investigate commutative rings. As a classical fact, the notion of regular local rings is defined by the number of minimal generators of the maximal ideal, and the numbers of minimal generators of powers of the maximal ideal play as a polynomial function. More generally, for each ideal, the number of generators of powers eventually agrees with a unique polynomial function. 
In particular, the function diverges infinitely. However, before being a polynomial, the numbers of generators of powers sometimes behave in an unexpected way even in polynomial rings (see for example \cite{AHZ} and \cite{Gasanova}). Therefore, it is natural to ask how bad the numbers of minimal generators of powers do there exist. 

The study of this direction may originate from the question of Judith Sally in 1974. She asked whether there exists a $1$-dimensional local domain for which the square of the maximal ideal has less generators than the maximal ideal itself. Such an example is given in \cite{HerzogWaldi}. Then many researchers study the numbers of generators of ideals $I$ and $I^2$. 
Especially, there are some interesting results for monomial ideals in $2$ variables.  Let $\mu(I)$ denote the  minimal number of generators of a monomial ideal $I$. We should note that it is not too difficult to construct examples of monomial ideals $I$ in a polynomial ring $S$ with at least $4$ variables such that $\mu(I)>\mu(I^2)$ since $I$ admits $\height (I)< \dim S$. Hence the assumption of $2$ variables is a reasonable restriction, and it actually provides ideals in arbitrary variables by the polynomial extension of ideals.

Let $K[x, y]$ be the polynomial ring over a field $K$ and $I$ a monomial ideal of $K[x, y]$. Then Eliahou, Herzog, and Saem \cite[Theorem 1.2]{EliahouHerzogSaem} showed a sharp lower bound $\mu(I^2)\ge 9$ when $\mu (I)\ge 6$, and Gasanova \cite{Gasanova} constructed monomial ideals $I$ such that $\mu(I)>\mu(I^n)$ for any positive integer $n$. 
With this background, our result is stated as follow:

\begin{Theorem}\label{b1.1}
For any positive integer $n$, there exists a monomial ideal $I$ satisfying the following conditions:
\begin{align*}
& \mu(I)>\mu(I^2)>\cdots >\mu(I^n)=(n+1)^2 \text{ and} \\ & \mu(I^k)=(n+2)k+1 \quad \text{for all $k\ge n$.}
\end{align*}
\end{Theorem}

Let us explain about this result. The inequalities $\mu(I)>\mu(I^2)>\cdots >\mu(I^n)$ are more unexpected than the behavior obtained in \cite{Gasanova}, and the equality $\mu(I^n)=(n+1)^2$ is related to the lower bound of \cite[Theorem 1.2]{EliahouHerzogSaem} regarding $(n+1)^2$ as $9$ by $n=2$. The latter assertion claims $\mu(I^k)$ agrees with the polynomial for $k\ge n$. Thus our construction provides one of the most unexpected behaviours for numbers of generators of monomial ideals before being a polynomial function.
We also note that our ideals to describe Theorem \ref{b1.1} are constructed by $n$ equigenerated ideals in contrast that \cite{EliahouHerzogSaem, Gasanova} and \cite{AHZ} basically explore ideals constructed by two equigenerated ideals.

In Section \ref{section2} we construct certain monomial ideals which yield Theorem \ref{b1.1}. We see that Theorem \ref{b1.1} also provides a peculiar example for the index of irreducibility in the sense of \cite{CQT}.
In what follows, let $S=K[x, y]$ be the polynomial ring over a field $K$ with $\deg(x)=\deg(y)=1$. Set $\fkm=(x, y)$. For a graded ideal $I$ of $S$, $\deg(I)$   denotes the smallest degree of homogeneous elements of $I$. We denote by $[I]_k$ the $K$-vector space spanned by the elements in $I$ of degree $k$. 
For  integers $a$ and $b$, $[a, b]$ denotes the set of all integers $i$ such that $a\le i \le b$.

\section{Proof of Theorem \ref{b1.1}}\label{section2}

Our main result is precisely stated as follow:

\begin{Theorem}\label{maintheorem}
Let $m, p_1, \dots, p_m, a_2, \dots, a_m\ge 2$ be positive integers.
Let 
\begin{align*}
I_1=&(x^{p_1}, y^{p_1})(x^{(m+1)p_1}, y^{(m+1)p_1}) \text{ and}\\
I_i=&x^{ip_1+p_i}{\cdot}y^{(m+2-i)p_1+p_2+\cdots+p_i}{\cdot}(x^{p_i}, y^{p_i})^{a_i-1}
\end{align*}
be ideals of $S$ for $2\le i \le m$. Set $I=I_1 + \dots+I_m$. Suppose that 
\begin{center}
$p_1=(a_i+1)p_i$ for $2\le i\le m$ and $p_2+\cdots +p_{m-1} < p_1$. 
\end{center}
Then 
\begin{enumerate}[{\rm (a)}] 
\item $I^k=\begin{cases}
I_1^{k-1}(I_1+\cdots+I_{m+1-k}) & \text{ if $1\le k \le m-1$}\\
I_1^k & \text{ if $m\le k$}.
\end{cases}$
\item $\mu(I^k)=\begin{cases}
(k+1)^2+k(a_2+\cdots+a_{m+1-k}) & \text{ if $1\le k \le m-1$}\\
(m+2)k+1 & \text{ if $m\le k$}.
\end{cases}$
\end{enumerate}

\end{Theorem}

Let us explain the idea of Theorem \ref{maintheorem}. By definition, $I$ consists of $m$ equigenerated ideals $I_1, \dots, I_m$. The assertion (a) claims that, as the power of $I$ gets bigger, the equigenerated ideals loses one by one. Hence, if we choose $I_i$ so that $\mu(I_i)\gg 0$, then $\mu(I^k)>\mu(I^{k+1})$ happens when $I_i$ no longer owes the monomial minimal generators of $I^{k+1}$.
The proof of Theorem \ref{maintheorem} proceeds in several steps. For simplicity we introduce some notations.

\begin{Notation}\label{b2.3}
Let $0\le a_1 <  \dots <a_\ell \le d$ be  integers.
We denote by $(a_1, \dots, a_\ell)_d$ the equigenerated monomial ideal generated by 
\[
x^{a_1}y^{d-a_1}, x^{a_2}y^{d-a_2}, \dots, x^{a_\ell}y^{d-a_\ell}.
\]
Similarly, ${}_K(a_1, \dots, a_\ell)_d$ denotes the $K$-vector space spanned by
\[
x^{a_1}y^{d-a_1}, x^{a_2}y^{d-a_2}, \dots, x^{a_\ell}y^{d-a_\ell}.
\]
We write $(a_1, \dots, a_\ell)=(a_1, \dots, a_\ell)_d$ and ${}_K(a_1, \dots, a_\ell)={}_K(a_1, \dots, a_\ell)_d$ if $a_\ell=d$.
\end{Notation}

The following lemma easily follows from Notation \ref{b2.3}.

\begin{Lemma}\label{a1.3}
For $1\le i \le m$, let $I_i$ be a monomial ideal of $S$ generated in single degree $d_i$.
Set $I_i=(a_{i1}, \dots, a_{i n_i})_{d_i}$ for  $1\le i \le m$. Then we have the following.
\begin{enumerate}[{\rm (a)}] 
\item $\prod_{i=1}^m I_i=(\sum_{i=1}^{m}a_{i k_i}\mid 1\le k_i \le n_i)_{d_1+\cdots+d_m}$.
\item Suppose that $d_1<\cdots<d_m\le \ell$. Then 
\[
\left[ I_1+\cdots +I_m\right]_\ell= {}_K(a_{i k_i}+t \mid 1\le i\le m, 1\le k_i \le n_i, 0\le t\le \ell-d_i)_\ell.
\]
\end{enumerate}
\end{Lemma}

In the rest of this paper, unless otherwise stated, we suppose the assumptions of Theorem \ref{maintheorem}. Set $d_u=\deg(I_u)$ for $1\le u\le m$.

\begin{Lemma}\label{a3.1}
\begin{enumerate}[{\rm (a)}] 
\item $d_1=(m+2)p_1$ and $d_u=d_1 + p_1 + p_2 + \dots +p_{u-1}$ for $2\le u \le m$.
\item $\deg(I_1^2)<\deg(I_1 I_2)<\cdots <\deg(I_1 I_m)$ and $\deg(I_1 I_m)<\deg(I_i I_j)$ for all $i,j\in [2, m]$.
\item $I_1^k=(s(m+1)p_1+tp_1 \mid s, t\in [0, k])_{kd_1}$ for $k>0$ and
\[
I_1^{k-1}I_u=(\{s(m+1)+t+u\}p_1+vp_u \mid s, t\in[0, k-1],  v\in [1, a_u])_{(k-1)d_1+d_u}
\]
for $2\le u \le m$.
\end{enumerate}

\end{Lemma}

\begin{proof}
(a) follows from the definition of $I_i$, and (b) follows from (a) and $p_2+\cdots+p_{m-1}<~p_1$.

(c) follows from the following computations:
\begin{align*}
I_1^k=&(0,p_1)^k(0,(m+1)p_1)^k\\
=&(0, p_1, 2p_1, \dots, kp_1)(0, (m+1)p_1, 2(m+1)p_1, \dots, k(m+1)p_1)\\
=&(s(m+1)p_1+tp_1 \mid s, t\in [0, k])_{kd_1}
\end{align*}
and
\begin{align*}
I_1^{k-1}I_u=&(0,p_1)^{k-1}(0,(m+1)p_1)^{k-1}x^{up_1+p_u}y^{(m+2-u)p_1+p_2+\cdots+p_u} (0,p_u)^{a_u-1}\\
=&(\{s(m+1)+t+u\}p_1+vp_u \mid s, t\in[0, k-1],  v\in [1, a_u])_{(k-1)d_1+d_u}.
\end{align*}
\end{proof}

\begin{Proposition}\label{a3.2}
Let $k>0$ and $2\le u \le m$. 
\begin{enumerate}[{\rm (a)}] 
\item For $m+1-k< u \le m$, $I_1^{k-1}(I_1+\dots +I_{u-1})\supseteq \fkm^{(k-1)d_1+d_u}$.
\item For $2\le u\le m+1-k$, 
\begin{align*}
&\left[ I_1^{k-1}(I_1+\cdots+I_{u})/  I_1^{k-1}(I_1+\cdots+I_{u-1})\right]_{(k-1)d_1+d_{u}}\\
\cong&{}_K{\left(\{s(m+1)+k+u-1\}p_1+t p_{u} \mid s\in [0, k-1], t\in [1, a_{u}] \right)}_{(k-1)d_1+d_{u}}.
\end{align*}
\end{enumerate}
\end{Proposition}

\begin{proof}
We prove by induction on $u$.
Assume that $u=2$. By Lemma \ref{a1.3}(b) and Lemma \ref{a3.1}(a) and (c),
\begin{align*}
[I_1^k]_{(k-1)d_1+d_2}&={}_K(s(m+1)p_1+tp_1+w' \mid s,t \in [0, k], w'\in [0, p_1])\\
&={}_K(s(m+1)p_1+w \mid s \in [0, k], w\in [0, (k+1)p_1])\\
[I_1^{k-1}I_2]_{(k-1)d_1+d_2}&={}_K(\{s(m+1)+t'+2\}p_1+vp_2 \mid s,t'\in [0, k-1], v\in [1, a_u]).
\end{align*}
Hence, if $m+1-k<u=2$($\Leftrightarrow m+1\le k+1$), then $I_1^k=\fkm^{(k-1)d_1+d_2}$. 
Suppose that $u=2\le m+1-k$. Then, for $0\le t'\le k-1$ and $1\le v\le a_2$,
\[
(t'+2)p_1+vp_2\le (k+1)p_1+a_2p_2<(k+2)p_1\le (m+1)p_1
\]
since $(a_2+1)p_2=p_1$. Hence 
\begin{align*}
&\left[ I_1^{k-1}(I_1+I_{2})/  I_1^{k}\right]_{(k-1)d_1+d_{2}}\\
\iso & {}_K{\left(\{s(m+1)+t'+2\}p_1+v p_{2} \middle| 
\begin{matrix}
s, t'\in[0, k-1], v\in [1, a_{2}], \text{ and }\\
(t'+2)p_1+vp_2>(k+1)p_1
\end{matrix}
\right)}_{(k-1)d_1+d_{2}}\\
=& {}_K{\left(\{s(m+1)+k+1\}p_1+v p_{2} \mid s\in[0, k-1], v\in [1, a_{2}]\right)}_{(k-1)d_1+d_{2}}.
\end{align*}

Assume that $u>2$ and our assertions (a) and (b) hold for $2, \dots, u-1$. Then, by induction hypothesis,
\begin{align}
&\left[ I_1^{k-1}(I_1+\cdots +I_{u-1})\right]_{(k-1)d_1+d_u} \nonumber \\
=&[I_1^{k}]_{(k-1)d_1+d_u} \nonumber \\
 & + \sum_{i=2}^{u-1} {}_K\left(s(m+1)p_1 + (k+i-1)p_1 + t_i p_i +w_i \middle| 
\begin{matrix}
s\in [0,k-1], t_i\in [1, a_i], \nonumber \\
w_i\in [0, d_u-d_i]
\end{matrix}
\right)\\
\begin{split}
=& {}_K(s(m+1)p_1 + w' \mid s\in [0, k], w'\in [0, kp_1 +d_u-d_1]) \\
 & + \sum_{i=2}^{u-1} {}_K\left(s(m+1)p_1 + (k+i-1)p_1 + t_i p_i +w_i \middle| 
\begin{matrix}
s\in [0,k-1], t_i\in [1, a_i], \\
w_i\in [0, d_u-d_i]
\end{matrix}
\right),
\end{split}
\end{align}
where the second equality follows from $(k-1)d_1+d_u-kd_1=d_u-d_1\ge p_1$ by Lemma \ref{a3.1} (a).

\begin{Claim}\label{claim1}
\begin{small}
\begin{align*}
\left[ I_1^{k-1}(I_1+\cdots + I_{u-1})\right]_{(k-1)d_1+d_u}
={}_K\left(s(m+1)p_1+w'' \mid s\in [0, k], w'' \in [0, (k+u-1)p_1]\right).
\end{align*}
\end{small}
\end{Claim}

\begin{proof}[Proof of Claim \ref{claim1}]
$(\subseteq)$: It is clear that $kp_1+d_u-d_1\le (k+u-1)p_1$. For $i\in [2,u-1], t_i\in [1, a_i], \text{ and } w_i\in [0, d_u-d_i]$, we have
\begin{align*}
(k+i-1)p_1+t_ip_i+w_i\le &(k+i-1)p_1 +a_ip_i+p_i+\cdots + p_{u-1}\\
=&(k+i)p_1 +p_{i+1}+\cdots + p_{u-1}\\
\le& (k+u-1)p_1.
\end{align*}
Hence we have the inclusion $\subseteq$ in Claim \ref{claim1}.

$(\supseteq)$: Let $0\le w''\le (k+u-1)p_1$. If $w''\le kp_1 +d_u-d_1$, then the monomial corresponding to $s(m+1)p_1+w''$ is in the former terms of (1).
Hence we may assume that $kp_1 +d_u-d_1< w''\le (k+u-1)p_1$. Then $(k+1)p_1 + p_2 \le w''$ since $u>2$ and Lemma \ref{a3.1}(a). Choose an integer $2\le i \le u-1$ so that 
\[
(k+i-1)p_1 + p_i\le w'' \le (k+i)p_1 + p_{i+1}.
\]
Here we regard $p_u$ as $0$ only this moment.
Then there exist integers $t\in [1, a_i]$ and $w\in [0,p_i+p_{i+1}]\subseteq[0, p_i+\dots+p_{u-1}]=[0, d_u-d_i]$ such that $w''=(k+i-1)p_1 + t p_i +w$. It follows that the inclusion $\supseteq$  holds.
\end{proof}

Assume that $m+2-k \le u\le m$. Then $(m+1)p_1\le (k+u-1)p_1$. Hence, by Claim \ref{claim1}, $I_1^{k-1}(I_1+\dots +I_{u-1})\supseteq \fkm^{(k-1)d_1+d_u}$. 

Suppose that $2< u\le m+1-k$. We have
\[
[I_1^{k-1}I_u]_{(k-1)d_1 + d_u}={}_K(\{s(m+1)+t+u\}p_1+vp_u \mid s, t\in[0, k-1],  v\in [1, a_u])_{(k-1)d_1+d_u}
\]
by Lemma \ref{a3.1}(c). Note that 
\begin{align*}
(t+u)p_1+vp_u\le&(k-1+u)p_1 + a_u p_u\\
<&(k-1+m+1-k)p_1+p_1\\
=&(m+1)p_1.
\end{align*}
Therefore, 
\begin{align*}
&\left[ I_1^{k-1}(I_1+\cdots+I_{u})/  I_1^{k-1}(I_1+\cdots+I_{u-1})\right]_{(k-1)d_1+d_{u}}\\
\cong&{}_K\left( \{s(m+1)+t+u\}p_1+vp_u \middle| 
\begin{matrix}
s, t\in[0, k-1],  v\in [1, a_u], \\
(t+u)p_1+vp_u>(k+u-1)p_1
\end{matrix}
\right)\\
=&{}_K\left( \{s(m+1)+k-1+u\}p_1+vp_u \middle| 
s\in[0, k-1],  v\in [1, a_u]
\right).
\end{align*}
\end{proof}

Now we are ready to prove Theorem \ref{maintheorem}.

\begin{proof}[Proof of Theorem \ref{maintheorem}]
(a): For $2\le k\le m-1$, by Lemma \ref{a3.1}(b),
\begin{align*}
I^k=&(I_1+\dots+I_m)^k\\
=&I_1^{k-1}(I_1+\dots+I_{m+1-k}) + (\text{monomials of degree} > (k-1)d_1 + d_{m+1-k}).
\end{align*}
On the other hand, $I_1^{k-1}(I_1+\dots+I_{m+1-k})\supseteq \fkm^{(k-1)d_1+d_{m+1-k}}$ by Proposition \ref{a3.2}(a). Hence, $I^k=I_1^{k-1}(I_1+\dots+I_{m+1-k})$.

If $k\ge m$, then $I_1^k=(sp_1 \mid s\in[0, k(m+2)])_{kd_1}=(x^{p_1}, y^{p_1})^{k(m+2)}$ by Lemma \ref{a3.1}(c). Hence 
$I_1^k\supseteq \fkm^{kd_1+p_1}=\fkm^{(k-1)d_1+d_2}$. It follows that $I^k=I_1^k$.

(b): Let $2\le k \le m-1$. Then 
\begin{align*}
\mu(I^k)=&\mu (I_1^{k-1}(I_1+\dots+I_{m+1-k}))\\
=&\mu(I^k)+\sum_{u=2}^{m+1-k} \dim_K \left[ I_1^{k-1}(I_1+\cdots+I_{u})/  I_1^{k-1}(I_1+\cdots+I_{u-1})\right]_{(k-1)d_1+d_{u}}\\
=& (k+1)^2 + \sum_{u=2}^{m+1-k} ka_u,
\end{align*}
where the third equality follows from Lemma \ref{a3.1}(c) and Proposition \ref{a3.2}(b).
For the case where $m\le k$, one can easily check that $\mu(I^k)=(m+2)k+1$ since $I^k=I_1^k$.
\end{proof}


Now we  prove Theorem \ref{b1.1}.

\begin{proof}[Proof of Theorem \ref{b1.1}]
Let $n\ge2$ be an integer. By Theorem \ref{maintheorem}, we have 
\begin{align*}
\mu(I^{n}) - \mu(I^{n-1}) &= 2n+1-(n-1)a_2 \\
\mu(I^{n-1}) - \mu(I^{n-2}) &= 2n-1+a_2-(n-2)a_3,\\
&\ \ \vdots\\
\mu(I^{k+1}) - \mu(I^{k}) &= 2k+3+a_2+\cdots +a_{m-k}-ka_{m+1-k},\\
&\ \ \vdots\\
\mu(I^{3}) - \mu(I^{2}) &= 7+a_2+\cdots+a_{n-2}-2a_{n-1},\text{ and}\\
\mu(I^{2}) - \mu(I) &= 5+a_2+\cdots+a_{n-1}-a_{n}.
\end{align*}
Therefore, by defining $a_2, a_3, \dots, a_n$ consecutively enough large, we get the assertion. We need to show that we can choose $a_2, \dots, a_n$ enough large with the assumptions 
\begin{center}
$p_1=(a_i+1)p_i$ for $2\le i\le n$ \quad and \quad $p_2+\cdots +p_{n-1} < p_1$. 
\end{center}
In fact, the former condition is always satisfied by choosing $p_1, \dots, p_n$ as follows.
\[
p_1=\prod_{i=2}^n (a_i+1) \quad \text{and} \quad p_i=\frac{p_1}{a_1+1} \quad \text{for $2\le i\le n$}.
\]
For the latter condition, by substituting $p_1=(a_i+1)p_i$, we can rephrase by 
\[
\frac{1}{a_2+1}+ \cdots + \frac{1}{a_{n-1}+1}<1.
\]
It follows that we can choose $a_2, \dots, a_n$ enough large as desired.
\end{proof}

\begin{Remark}
In the proof of Theorem \ref{b1.1}, for two integers $1\le i< j \le n$, we can choose 
$a_i=a_j=2$. It follows that at most two of inequalities in Theorem \ref{b1.1}(a) can change to the converse inequalities.
\end{Remark}
\begin{Example}
With the notation of Theorem \ref{maintheorem}, set 
\begin{small}
\begin{center}
$m=5$, $p_1=72$, $p_2=18$, $p_3=12$, $p_4=8$, $p_5=2$, $a_2=3$, $a_3=5$, $a_4=8$, and $a_5=35$. 
\end{center}
\end{small}
Then $I=I_1+I_2+I_3+I_4+I_5$, where 
\begin{align*}
&I_1=(x^{72},y^{72})(x^{432},y^{432}), &&I_2=x^{162}y^{378}(x^{18},y^{18})^2, &&I_3=x^{228}y^{318}(x^{12},y^{12})^4,\\
&I_4=x^{296}y^{254}(x^8,y^8)^7,  \text{ and } && I_5=x^{362}y^{184}(x^2,y^2)^{34}. && 
\end{align*}
One can check by CoCoA \cite{Cocoa} that $\mu (I)=55$, $\mu (I^2)=41$, $\mu (I^3)=40$, $\mu (I^4)=37$, $\mu (I^5)=36=6^2$, and $\mu (I^6)=43$.
\end{Example}

\begin{Question}
\begin{enumerate}[{\rm (a)}] 
\item For any positive integer $n$ and $(n-1)$ signs $\alpha_1, \dots, \alpha_{n-1}\in \{+, -\}$, does there exist a monomial ideal $I$ satisfying the following conditions?
\begin{enumerate}[{\rm (i)}] 
\item $\sign(\mu(I^{k+1})-\mu(I^k))=\alpha_k$ for all $1\le k\le n-1$.
\item $\mu(I^k)=(n+2)k+1$ for all $k\ge n$.
\end{enumerate}
\item In reference to \cite[Theorem 1.2]{EliahouHerzogSaem}, 
for any monomial ideal $I$ with $\mu(I)\gg 0$, does $\mu(I^n)\ge (n+1)^2$ hold?
\end{enumerate}
\end{Question}

Theorem \ref{b1.1} also gives a peculiar example for the index of irreducibility.
For a while, let $R$ be a Noetherian local  ring and $I \subset R$ be an $\mm$-primary ideal.  Then Cuong, Quy, and Truong \cite[Theorem 4.1]{CQT} showed that the numerical function
\begin{align*}
p: \NN  \rightarrow    \NN, \ n\mapsto  \mathrm{r}(R/I^n)
\end{align*}
is a polynomial function of degree $\dim R-1$ for $n \gg 0$.
Here, $\mathrm{r}(R/I^n)$ denotes the Cohen-Macaulay type of $R/I^n$ (or the {\it index of irreducibility}  of $I^n$ in the sense of \cite{CQT}), that is the length $\ell_R((I:_R \mm)/I^n$)  of the socle of $R/I^n$. 
In the papers \cite{CQT, T}, they gave a nice characterization of the Cohen-Macaulay property by the index of irreducibility of parameter ideals. On the other hand, little is known about the above numerical function other than parameter ideals. In the rest of this paper, we will discuss the relationship between the Cohen-Macaulay type and the number of minimal generators. For a finitely generated $R$-module $M$, let $\mu(M)$ denotes the minimal number of generators of $M$.

\begin{Proposition}\label{type}
Let $(R, \fkm)$ be a $2$-dimensional  regular local ring and $I \subset R$ be an $\mm$-primary ideal. Then we have $\mu(I^n)=\mathrm{r}(R/I^n)+1$.
\end{Proposition}

\begin{proof}
Let 
\[
\begin{CD}
 \FF \:\; 0  @> \varphi_{3} >> F_{2} @> \varphi_{2} >> F_1@>\varphi_1 >> F_0 @>\varphi_0 >> R/I^n\to  0
\end{CD}
\]
be the minimal $R$-free resolution of $R/I^n$. Note that $\rank F_1 = \mu (I^n)$. It follows that 
\begin{align} \label{rank}
 \rank F_2=\rank F_1-\rank F_0= \mu (I^n) -1.
\end{align}
On the other hand, by applying the $R$-dual $(-)^*=\Hom_R(-, R)$ to $\FF$, we get the complex
\[
\begin{CD}
 \FF^* \:\; 0 \to F_{0}^* @> \varphi_{1}^* >> F_1^* @>\varphi_2^* >> F_2^* @>\varphi_3^*>>    0.
\end{CD}
\]
The second homology of the complex is $\Ext_R^2(R/I^n, R)$, which is the canonical module of $R/I^n$ by \cite[Theorem~3.3.7]{BH}. Hence we obtain that 
\[
\mathrm{r}(R/I^n)=\mu(\Ext_R^2(R/I^n, R))=\mu(F_2^*/\Im \varphi_2^*)=\rank F_2^*
\]
by \cite[Proposition~3.3.11]{BH} and the fact that $\FF$ is minimal. It follows that our assertion holds by (\ref{rank}).
\end{proof}

\begin{Corollary}
Let $S=K[x,y]$ be a polynomial ring over a field $K$. Then we have the following.
\begin{enumerate}[{\rm (a)}] 
\item Let $I$ be a monomial ideal. Then we have $\mathrm{r}(S/I^2)\ge 8$ if $\mathrm{r}(S/I)\ge 5$.
\item For any positive integer $n$, there exists a monomial ideal $I$ such that 
\begin{align*}
 \mathrm{r}(S/I)>\mathrm{r}(S/I^2)>\cdots >\mathrm{r}(S/I^n)=(n+1)^2-1.
\end{align*}
\end{enumerate} 
\end{Corollary}
\begin{proof}
(a) follows from \cite[Theorem 1.2]{EliahouHerzogSaem} and (b) follows from Theorem ~\ref{b1.1}.
\end{proof}

\begin{Acknowledgment}
Most of this work was done during the visit of the authors to the University of Duisburg-Essen in 2019-2020. The authors are grateful to J\"{u}rgen Herzog for telling them this topic and to the referee for kind advice and comments.
\end{Acknowledgment}


\end{document}